\documentclass[conference]{IEEEtran}
\usepackage[nocompress]{cite}

\usepackage{amsmath,amssymb,amsfonts,amsthm}
\usepackage{graphicx}
\usepackage{textcomp}

\usepackage{mathdots}
\usepackage{mathtools}
\usepackage{newfloat}
\usepackage{xfrac}
\usepackage{booktabs}

\DeclarePairedDelimiter{\norm}{\lVert}{\rVert}

\newcommand{\invnb}[1]{#1^{-1}}

\newtheorem{theorem}{Theorem}
\newtheorem{lemma}{Lemma}

\newtheorem{definition}{Definition}

\newcommand\extrafootertext[1]{%
    \bgroup
    \renewcommand\thefootnote{\fnsymbol{footnote}}%
    \renewcommand\thempfootnote{\fnsymbol{mpfootnote}}%
    \footnotetext[0]{#1}%
    \egroup
}

\begin{document}
%
\title{Conditions for Digit Stability in Iterative Methods Using the Redundant Number Representation}

\author{\IEEEauthorblockN{Ian McInerney}
\IEEEauthorblockA{\textit{Department of Mathematics, The University of Manchester}\\
Manchester, UK\\
Email: ian.mcinerney@manchester.ac.uk}}


\maketitle

\begin{abstract}
Iterative methods play an important role in science and engineering applications, with uses ranging from linear system solvers in finite element methods to optimization solvers in model predictive control.
 Recently, a new computational strategy for iterative methods called ARCHITECT was proposed by Li et al.\ in \cite{liARCHITECTArbitraryPrecisionHardware2020} that uses the redundant number representation to create ``stable digits'' in the Most-significant Digits (MSDs) of an iterate, allowing the future iterations to assume the stable MSDs have not changed their value, eliminating the need to recompute them.
 In this work, we present a theoretical analysis of how these ``stable digits'' arise in iterative methods by showing that a Fej\'er monotone sequence in the redundant number representation can develop stable MSDs in the elements of the sequence as the sequence grows in length.
 This property of Fej\'er monotone sequences allows us to expand the class of iterative methods known to have MSD stability when using the redundant number representation to include any fixed-point iteration of a contractive Lipschitz continuous function.
 We then show that this allows for the theoretical guarantee of digit stability not just in the Jacobi method that was previously analyzed by Li et al.\ in \cite{liDigitStabilityInference2021}, but also in other commonly used methods such as Newton's method.
\end{abstract}



\extrafootertext{For the purpose of open access, the author has applied a Creative Commons Attribution (CC BY) licence to any Author Accepted Manuscript version arising.}

\section{Introduction}

The redundant number system plays a key role in computer arithmetic, with uses ranging from carry-free addition, multiplication and division algorithms \cite{ercegovacDigitalArithmetic2004} to Most-significant Digit (MSD) first operations such as Online arithmetic \cite{ercegovacOnlineArithmeticDSP1989} or the E-method for evaluating polynomials \cite{ercegovacGeneralHardwareOrientedMethod1977}.
 Recent work has extended MSD first arithmetic to the implementation of iterative methods, with the ARCHITECT framework proposed by Li et al.\ in \cite{liARCHITECTArbitraryPrecisionHardware2020} showing large speed-ups in FPGA implementations of the Jacobi method by exploiting the ideas of ``don't-change'' stable MSDs and ``don't-care'' Least-significant Digits (LSDs) to reduce the number of total digits computed in each iteration.
 Prior work to analyze and exploit the stable MSDs, which are the MSDs that don't change their digit value in any future iteration, has been algorithm-specific, with works such as \cite{liDigitStabilityInference2021} and \cite{ercegovacGeneralHardwareOrientedMethod1977} starting with a specific algorithm and then showing that its iterate sequence has MSD stability.

In this paper, we instead propose to reverse the order of the analysis by finding sequences of numbers in the redundant number system with stable MSDs, and then finding iterative methods that have those sequences as iterates.
 We start by showing that a Fej\'er monotone sequence can have stable MSDs when represented using redundant numbers, and we then use that result to show that any iterative method whose iterates are a Fej\'er monotone sequence (such as fixed-point iterations of contractive Lipschitz continuous functions) can have digit stability.
 This allows us to derive theoretical guarantees for the existence of stable MSDs in the iterate sequence of both the Jacobi method (which was previously analyzed in \cite{liDigitStabilityInference2021}), and Newton's method (which was only experimentally shown to have stable digits in \cite{liARCHITECTArbitraryPrecisionHardware2020})

\section{Background}

\subsection{Redundant Number Representation}

A redundant number representation is one where a real number can be represented by multiple different digit patterns, with a commonly used representation being the signed-digit representation \cite{avizienisSignedDigitNumbeRepresentations1961}, where the digits of a radix $r$ number using a symmetric signed-digit representation can take any value from the set
 \begin{equation*}
     S \coloneqq \{-\gamma, \dots, -1, 0, 1, \dots, \gamma \},
 \end{equation*}
 where $\sfrac{r}{2} \leq \gamma \leq r-1$.
 The choice of $\gamma$ determines the level of redundancy in the representation, with the choice $\gamma = r-1$ known as a \textit{maximally redundant} representation.

A key difference of the symmetric maximally redundant number representation compared to the standard number representation is in the real numbers that can be represented by just appending/changing the least significant digits of a number, a property known as the \textit{representation interval}.
 In the standard representation, the representation interval is single-sided, with changes in the least-significant digits only able to increase the value of the real number being represented.
 However, a redundant representation introduces a dual-sided representation interval that is able to both increase and decrease the value of the represented real number by simply adding/changing least-significant digits, with Li et al.\ \cite{liDigitStabilityInference2021} deriving the representation interval for a symmetric maximally redundant representation, shown here in Lemma~\ref{lem:redundantInterval}.
\begin{lemma}[Representation interval \cite{liDigitStabilityInference2021}]
    \label{lem:redundantInterval}
	Let $x$ be a $D$-digit number in the symmetric maximally redundant signed-digit redundant number system with radix $r$.
	If additional digits are appended to $x$ to form a new number, $\tilde{x}$, then
    \begin{equation*}
        \tilde{x} \in \left(x - r^{-D},~x + r^{-D}\right).
    \end{equation*}
\end{lemma}

\subsection{Fej\'er Monotonicity and Nonexpansive Operators}

Fej\'er monotone sequences play an important role in the analysis and understanding of iterative methods, especially in fields involving monotone operators such as convex optimization \cite{bauschkeConvexAnalysisMonotone2011}.
 At its core, a Fej\'er monotone sequence is one where the distance between the elements of the sequence and a given set is not increasing as the sequence grows in length, which we define formally in Definition~\ref{def:fejermonotonicity}.
\begin{definition}[Fej\'er monotonicity {\cite[\S 5.1]{bauschkeConvexAnalysisMonotone2011}}]
    \label{def:fejermonotonicity}
    Let $\mathcal{C}$ be a nonempty subset of a Hilbert space $\mathcal{H}$ and let $\left(x^{(n)}\right)_{n \in \mathbb{N}}$ be a sequence in $\mathcal{H}$. Then $\left(x^{(n)}\right)_{n \in \mathbb{N}}$ is \textit{Fej\'er monotone} with respect to $\mathcal{C}$ if
    \begin{equation}
        \label{eq:monotonesequence}
        \norm{x^{(n+1)} - x} \leq \norm{x^{(n)} - x} \quad (\forall x \in \mathcal{C})(\forall n \in \mathbb{N}).
    \end{equation}
\end{definition}

Closely related to the idea of Fej\'er monotone sequences is that of monotone operators, and we specifically focus on nonexpansive and contractive operators in this work, which we define in Definition~\ref{def:nonexpansiveoperator}.
\begin{definition}[Nonexpansive/contractive operators {\cite{ryuPrimerMonotoneOperator2016}}]
    \label{def:nonexpansiveoperator}
    Let $\mathcal{D}$ be a nonempty subset of the Hilbert space $\mathcal{H}$ and let $T(\cdot) : \mathcal{D} \to \mathcal{H}$ be Lipschitz continuous with Lipschitz constant $L$, i.e.\
    \begin{equation}
        \label{eq:lipschitzoperator}
        \norm{T(x) - T(y)} \leq L \norm{x - y} \quad (\forall x \in \mathcal{D}) (\forall y \in \mathcal{D}).
    \end{equation}
    Then
    \begin{itemize}
        \item if $L = 1$, $T(\cdot)$ is a \textit{nonexpansive} opertor, or
        \item if $L < 1$, $T(\cdot)$ is a \textit{contractive} operator.
    \end{itemize}
\end{definition}

Fej\'er monotone sequences and nonexpansive operators are very closely linked, with the iterate sequence for a fixed-point iteration of a nonexpansive operator generating a Fej\'er monotone sequence with respect to the fixed-point, as shown in Lemma~\ref{lem:operatorSequence}.

\begin{lemma}
    \label{lem:operatorSequence}
    Let $T(\cdot)$ be a nonexpansive/contractive operator (i.e.\ $L \leq 1$) and $\mathcal{F}$ be the nonempty set of fixed points of $T(\cdot)$, then the sequence of iterates $\left(x^{(n)}\right)_{n \in \mathbb{N}}$ of the iteration
    \begin{equation}
        \label{eq:lem:operatorrelation}
        x^{(n+1)} = T\left(x^{(n)}\right) \qquad (\forall n \in \mathbb{N})
    \end{equation}
    is Fej\'er monotone with respect to $\mathcal{F}$.
\end{lemma}
\begin{proof}
    This follows trivially from Definition~\ref{def:nonexpansiveoperator} by choosing $y \in \mathcal{F}$ and substituting~\eqref{eq:lem:operatorrelation} into~\eqref{def:nonexpansiveoperator}, giving
    \begin{equation*}
        \norm{x^{(n+1)} - y} \leq L \norm{x^{(n)} - y} \quad (\forall x \in \mathcal{C})(\forall n \in \mathbb{N}).
    \end{equation*}
    Since $L \leq 1$, this satisfies inequality~\eqref{eq:monotonesequence} and the sequence of iterates $\left(x^{(n)}\right)_{n \in \mathbb{N}}$ will be Fej\'er monotone.
\end{proof}

\section{Fej\'er Monotonicity of Redundant Numbers}
We now begin our analysis of the redundant number system by noting the similarity between the Fej\'er monotone sequences in Definition~\ref{def:fejermonotonicity} and the representation interval in Lemma~\ref{lem:redundantInterval}, which we turn into an existence guarantee for digit stability of Fej\'er monotone sequences in Theorem~\ref{thm:fejerRedundancy}.
 We will be carrying out our analysis using the infinity norm in Definitions~\ref{def:fejermonotonicity} and~\ref{def:nonexpansiveoperator}, since when examining digit stability, we want to look at the largest error between the iterates and their limit point.
\begin{theorem}
    \label{thm:fejerRedundancy}
    Let $\left(x^{(k)}\right)_{k \in \mathbb{N}}$ be a Fej\'er monotone sequence in the symmetric maximally redundant number representation of radix $r$ and $x$ be the only point in the set $\mathcal{C}$. If element $n$ of the sequence satisfies
    \begin{equation}
        \label{eq:fejer:Dcondition}
        \norm{x^{(n)} - x}_{\infty} < r^{-D},
    \end{equation}
    for some positive integer $D$, then there exists a representation of the sequence where the $D$ MSDs of all future elements with index $k$ satisfy
    \begin{equation}
        \label{eq:fejer:digitStabilityCondition}
        x^{(k)}_{i} = x^{(n)}_{i} \quad (\forall i \in \{ 1, \dots, D \}) (\forall k > n).
    \end{equation}
\end{theorem}
\begin{proof}
    Since we are using the infinity norm, we can without loss of generality focus on the 1-dimensional case (e.g.\ just the element in the vector $x^{(n)}$ that is the furthest from the corresponding element in $x$).
    From the infinity norm, we know that if $x^{(n)}$ satisfies condition~\eqref{eq:fejer:Dcondition}, then the point $x$ lies in the interval
    \begin{equation}
        \label{eq:fejer:xinterval}
        x \in \left(x^{(n)} - r^{-D},~x^{(n)} + r^{-D}\right).
    \end{equation}
    Since $x^{(n)}$ is in the redundant number representation, we see that the interval~\eqref{eq:fejer:xinterval} exactly matches the representation interval given in Lemma~\ref{lem:redundantInterval}, meaning that the redundant representation of $x^{(n)}$ can represent $x$ by only appending digits and not changing the first $D$ MSDs.
    
    We know that since $x^{(n)}$ and $x^{(n+1)}$ belong to a Fej\'er monotone sequence, they will satisfy the inequality
    \begin{equation}
        \label{eq:fejer:monotonerelation}
        \norm{x^{(n+1)} - x} \leq \norm{x^{(n)} - x} < r^{-D},
    \end{equation}
    which means that $x^{(n+1)}$ is also contained within the representation interval of $x^{(n)}$ given by \eqref{eq:fejer:xinterval}, meaning $x^{(n+1)}$ can be represented by only appending new digits to $x^{(n)}$.
    The Fej\'er montonicity means that relation~\eqref{eq:fejer:monotonerelation} will also hold for all iterates after element $n$, meaning we can represent all future elements by simply appending new digits to $x^{(n)}$, leading to the digit stability condition in~\eqref{eq:fejer:digitStabilityCondition}.
\end{proof}

Note that the proof for Theorem~\ref{thm:fejerRedundancy} is simpler than the equivalent proof for Lemma~3 by Li et al.\ in \cite{liDigitStabilityInference2021}, since we can make use of the Fej\'er montonicity of the sequence instead of having to use algorithm-specific information about how the iterates are generated and the intervals will change.

The digit stability result presented in Theorem~\ref{thm:fejerRedundancy} is a direct result of having the redundancy in the number representation.
 In the standard representation, the representation intervals will be single-sided and encompass only
 \begin{equation*}
     \tilde{x} \in \left[ x, x + r^{-D} \right),
 \end{equation*}
 while the relation between the elements of a  Fej\'er monotone sequence uses normwise relations that are double-sided.
 This means that when a Fej\'er monotone sequence converges to a limit point with a finite number of decimal places, any oscillation around the limit point during the sequence will require changing the MSDs in the standard representation, but the redundant representation is able to accommodate for these oscillations using only the future LSDs.
 
An example of this oscillation can be seen in the Fej\'er monotone sequence with a limit point of $0.5$ given in Table~\ref{tab:fejerExample}.
 The sequence oscillates around $0.5$, and the standard (base 10 and binary) representations cannot stabilize the MSDs in the sequence, while the symmetric maximally redundant representation in radix-2 can stabilize the MSDs in the presence of the oscillation.
 This example also showcases the requirement that inequality~\eqref{eq:fejer:Dcondition} in Theorem~\ref{thm:fejerRedundancy} must be satisfied to generate stable digits.
 While we have two stable digits starting at element two, we do not generate the next stable digit until after element four, since the distance to 0.5 for elements two and three are both strictly greater than $2^{-2}$.
 
The digit sequences for values of a Fej\'er monotone sequence in the redundant representation are not unique though, especially in the symmetric maximally redundant representation we use in this work.
 In a symmetric maximally redundant representation, a value can always be represented using its original (non-redundant) digit sequence as well as a sequence utilizing the redundant digits.
 This can also be seen in the example sequence in Table~\ref{tab:fejerExample}, since the column listing the binary representation utilizes a strict subset of the digits used in the redundant column (i.e.\ using only $\{0, 1\}$).
 Therefore, Theorem~\ref{thm:fejerRedundancy} simply says that such a sequence using the redundant representation exists, not that the Fej\'er monotone sequence will always be represented by it.
 
\begin{table}[t]
 \centering
 \caption{Sample Fej\'er monotone sequence that converges to $0.5$. Redundant representation uses radix-2 with the symmetric maximal digit set $\{-1, 0, 1\}$.}
 \label{tab:fejerExample}
 \begin{tabular}{c|l|l|l|l}
      Element \# & Standard & Binary & Redundant & $\norm{x_{i} - 0.5}$ \\
      \hline
      1 & $1.0$      & $1. 0 0 0 0 0 0$ & $1. 0 0 0$ & $0.5$\\
      2 & $0.125$    & $0. 0 0 1 0 0 0$ & $1. \bar{1} \bar{1} \bar{1}$ & $0.375$\\
      3 & $0.75$     & $0. 1 1 0 0 0 0$ & $1. \bar{1} 1 0$ & $0.25$\\
      4 & $0.375$    & $0. 0 1 1 0 0 0$ & $1. \bar{1} 0 \bar{1}$ & $0.125$\\
      5 & $0.5625$   & $0. 1 0 0 1 0 0$ & $1. \bar{1} 0 0 1$ & $0.0625$\\
      6 & $0.46875$  & $0. 0 1 1 1 1 0$ & $1. \bar{1} 0 0 0 \bar{1}$ & $0.03125$\\
      7 & $0.515625$ & $0. 1 0 0 0 0 1$ & $1. \bar{1} 0 0 0 0 1$ & $0.015625$\\
      $\vdots$ & $\vdots$ & $\vdots$ & $\vdots$ & $\vdots$
 \end{tabular}
\end{table}

\section{MSD Stability in Iterative Methods}

Now that we have shown that a Fej\'er monotone sequence can have digit stability in the redundant representation, we can easily say that any fixed-point iteration in the redundant number representation that generates a sequence of iterates that are Fej\'er monotone can have digit stability of its iterates when they satisfy condition~\eqref{eq:fejer:Dcondition}.

\begin{theorem}
    \label{thm:fejerIteration}
    Let $f(\cdot)$ be a function that has a fixed-point of $x^{*}$ and is used in the fixed-point iteration $x^{(k+1)} = f(x^{(k)})$ with the redundant number representation to generate the Fej\'er monotone sequence of iterates $\left(x^{(k)}\right)_{k \in \mathbb{N}}$.
    If element $n$ of the sequence generated by $f(\cdot)$ satisfies
    \begin{equation*}
        \norm{x^{(n)} - x^{*}}_{\infty} < r^{-D},
    \end{equation*}
    for some positive integer $D$, then there exists a representation of the iterate sequence where the $D$ MSDs of all future elements with index $k$ satisfy
    \begin{equation*}
        x^{(k)}_{i} = x^{(n)}_{i} \quad (\forall i \in \{ 1, \dots, D \}) (\forall k > n).
    \end{equation*}
\end{theorem}
\begin{proof}
    This follows directly from the result in Theorem~\ref{thm:fejerRedundancy}.
\end{proof}

While Theorem~\ref{thm:fejerIteration} is a nice and simple result, we can now also derive more specific results for classes of iterative methods, such as Lipschitz continuous methods in Theorem~\ref{thm:lipschitzIteration}.
\begin{theorem}
    \label{thm:lipschitzIteration}
    Let $f(\cdot)$ be a function that is locally Lipschitz inside the set $\chi$ with Lipschitz constant $L_{f} < 1$.
    If $f(\cdot)$ has a fixed-point of $x^{*} \in \chi$ and is used in the fixed-point iteration $x^{(k+1)} = f(x^{(k)})$ with the redundant number representation to generate the sequence of iterates $\left(x^{(k)}\right)_{k \in \mathbb{N}} \in \chi$, then given a number of digits $D \geq 1$ there exists an element index $\hat{k} \geq 0$ such that
    \begin{equation}
        \label{eq:msd:digstability}
        x^{(k)}_{i} = x^{(\hat{k})}_{i} \quad (\forall i \in \{ 1, \dots, D \}) (\forall k > \hat{k}).
    \end{equation}
\end{theorem}
\begin{proof}
    We begin by noting that since $f(\cdot)$ is locally Lipschitz inside $\chi$ with a Lipschitz constant $L_{f} < 1$, $f(\cdot)$ will be a contractive operator inside $\chi$ and Lemma~\ref{lem:operatorSequence} says that the iterates $\left(x^{(k)}\right)_{k \in \mathbb{N}}$ of $f$ will form a Fej\'er monotone sequence that stays in the set $\chi$.
    Since $L_{f}$ is strictly less than 1, the residuals between the iterates $x^{(k)}$ and the fixed point $x^{*}$ will be a decreasing sequence that can be bounded using a function of the Lipschitz constant and the difference between the two initial iterates \cite[Thm. 4.1]{suliIntroductionNumericalAnalysis2003}
    \begin{equation}
        \label{eq:msd:seqbound}
        \norm{x^{(k)} - x^{*}} \leq L^{k} \frac{1}{1 - L} \norm{x^{(1)} - x^{(0)}},
    \end{equation}
    with $\lim_{k \to \infty} L^{k} = 0$.
    This means that the residual between the iterates $x^{(k)}$ and $x^{*}$ will eventually be less than $r^{-D}$ for a given $D \geq 1$ and $r \geq 2$, satisfying the condition in Theorem~\ref{thm:fejerRedundancy} and allowing it to be used to create the digit stability result~\eqref{eq:msd:digstability}.
\end{proof}

A key point to notice in Theorems~\ref{thm:fejerIteration} and~\ref{thm:lipschitzIteration} is that they are concerned only with the \textit{existence} of an iterate sequence with stable MSDs, and do not comment on how such a sequence can be constructed.
 The construction of the sequence with stable MSDs will still be application/algorithm-specific, with the designer needing to ensure that the computations performed in the fixed-point iteration will preserve the MSD stability as they are performed, such as through the use of online arithmetic or the ARCHITECT system~\cite{liARCHITECTArbitraryPrecisionHardware2020}.

\section{Example Iterative Methods with MSD Stability}

Now that we have presented conditions for when the MSDs of iterates can be stable, we turn our attention to showing two example iterative methods where we can apply these conditions and derive a guarantee for digit stability.

\subsection{Stationary iterative linear solvers}

A common iterative method examined by prior work on redundant number systems (e.g.\ \cite{liDigitStabilityInference2021} and~\cite{ercegovacGeneralHardwareOrientedMethod1977}) is the stationary iterative method for solving linear systems of the form $Ax = b$, which has the fixed-point iteration
\begin{equation*}
    x^{(k+1)} = T\left(x^{(k)}\right) \coloneqq (I - \invnb{M} A)x^{(k)} + \invnb{M}b.
\end{equation*}
 Common forms of this method are the Jacobi method when $M$ is the diagonal of $A$, the Gauss-Seidel method when $M$ is the lower triangular portion of $A$, and the successive overrelaxation (SOR) method when $M = \invnb{\omega} D - \bar{L}$ \cite[Chap. 2]{Greenbaum1997}.
 To begin our analysis, we examine the difference between two successive function values
 \begin{multline*}
    \norm{T\left(x^{(k+1)}\right) - T\left(x^{(k)}\right)} = \\
     \norm{ (I - \invnb{M} A)x^{(k+1)} + \invnb{M}b - (I - \invnb{M} A)x^{(k)} - \invnb{M}b}.     
 \end{multline*}
 Simplifying this, we find that it is equal to
\begin{equation*}
     \norm{ (I - \invnb{M} A) (x^{(k+1)} - x^{(k)})},
 \end{equation*}
 and by applying the triangle inequality we find that
 \begin{equation*}
     \norm{T\left(x^{(k+1)}\right) - T\left(x^{(k)}\right)} \leq 
     \norm{I - \invnb{M}A} \norm{x^{(k+1)} - x^{(k)}}.
 \end{equation*}
 This inequality is equivalent to a Lipschitz condition for the stationary iterative solver, with the Lipschitz constant
 \begin{equation*}
     L = \norm{I - \invnb{M}A}.
 \end{equation*}
 
We now can apply Theorem~\ref{thm:lipschitzIteration} to this method to find out that the simple iterative linear solvers will have digit stability when using the redundant number system, provided the condition
 \begin{equation*}
     \norm{I - \invnb{M}A}_{\infty} < 1
 \end{equation*}
 is satisfied, which is the same condition derived by Li et al.\ in \cite{liDigitStabilityInference2021}.

\subsection{Newton's method}

Another popular iterative method utilized in many fields is Newton's method for finding the roots of the function $f(\cdot)$, which is given by the fixed-point iteration (for 1-dimension)
 \begin{equation}
    \label{eq:newtons}
    x^{(k+1)} = T\left(x^{(k)}\right) \coloneqq x^{(k)} - \frac{f(x^{(k)})}{f'(x^{(k)})},
 \end{equation}
 where $f'(\cdot)$ is the derivative of $f$.
 We can manipulate the Lipschitz continuous condition in Definition~\ref{def:nonexpansiveoperator} to be
 \begin{equation*}
     \frac{\norm{T(x^{(n+1)}) - T(x^{(n)})}}{\norm{x^{(n+1)}) - x^{(n)}}} \leq L,
 \end{equation*}
 which for the 1-dimensional case can be translated into a condition on the derivative of $T(\cdot)$ that $T'(x) \leq L$, meaning a contractive operator has $T'(x) < 1$ for all $x$ in its region of convergence \cite[Chap. 1]{suliIntroductionNumericalAnalysis2003}.

We then know for the Newton's method~\eqref{eq:newtons}, that its derivative will be
 \begin{equation*}
    T'\left(x\right) \coloneqq \frac{f(x) f''(x)}{\left( f'(x) \right)^2}.
 \end{equation*}
 Since Newton's method needs $T'(x) < 1$ (and therefore $L < 1$) for its convergence, we can use Theorem~\ref{thm:lipschitzIteration} to say that as the Newton's method converges, there exists a redundant number representation of its iterates with stable digits.

\section{Conclusion}

In this work, we examined the conditions needed for there to be MSD stability in sequences represented in the redundant number system.
 We did this through two key results: showing that a Fej\'er monotone sequence can have a redundant representation with MSD stability, and then using that result to show that an iterative method that generates a Fej\'er monotone sequence of iterates can have MSD stability when using the redundant number system.

These results only focused on the existence of these sequences, though, and did not discuss how to identify or create them.
 These are both key points in the successful use of these results in real-world applications, so future work should explore them in more detail.
 For instance, future work can explore if there is a redundant representation where a Fej\'er monotone sequence is guaranteed to have digit stability in all its representations (e.g.\ non-symmetric or non-maximal redundancy), and also how the representation is preserved during the computations in an iterative method (e.g.\ will the redundant sequence be preserved).
 Additionally, the results linking Lipschitz continuous functions with digit stable redundant sequences can be further exploited to try to predict the number of stable MSDs that will appear in each iteration, allowing for more general results than those previously presented in~\cite{liDigitStabilityInference2021}.

\section*{Acknowledgment}

We would like to thank Dr. He Li for reading and commenting on an early version of this manuscript.
This work was supported by Engineering and Physical Sciences Research Council grant EP/P020720/1.


\bibliography{main}

\end{document}